\documentclass{amsart}

\usepackage{amssymb,amscd,amsmath,hyperref,color,enumerate}
\usepackage[all,cmtip]{xy}

\def\R{\mathbb{R}}
\def\C{\mathbb{C}}
\def\N{\mathbb{N}}
\def\Z{\mathbb{Z}}
\def\F{\mathbb{F}}
\def\cL{\mathcal{L}}
\def\cR{\mathcal{R}}

\def\s{\sigma}

\def\d{\delta}

\def\r{\rho}
\newcommand\CB{\mathop{\mathrm{CB}}\nolimits}
\def\El{\mathcal{E}\ell}
\def\MM{\mathbb{M}}
\newcommand{\Ad}{\mathop{\mathrm{Ad}}}
\newcommand{\M}{M_{\mathrm{loc}}}
\newcommand{\PP}{\mathrm{Prim}}
\def\CE{\mathcal{E}}

\newtheorem{proposition}{Proposition}[section]
\newtheorem{lemma}[proposition]{Lemma}

\newtheorem{theorem}[proposition]{Theorem}
\newtheorem{corollary}[proposition]{Corollary}
\theoremstyle{definition}
\newtheorem{remark}[proposition]{Remark}

\newtheorem{problem}[proposition]{Problem}
\newtheorem{definition}[proposition]{Definition}
\newtheorem{example}[proposition]{Example}

\numberwithin{equation}{section}

\begin{document}

\title[]{The cb-norm approximation of generalized skew derivations by elementary operators}

\author{Ilja Gogi\'{c}}

\date{3 July 2019}

\dedicatory{Dedicated to the memory of my mentor and a friend, Professor R.~M.~Timoney}

\address{Department of Mathematics, University of Zagreb, Bijeni\v{c}ka 30,
10000 Zagreb, Croatia}

\email{ilja@math.hr}

\thanks{This work was supported by the Croatian Science Foundation under the project IP-2016-06-1046.}



\begin{abstract} 
Let $A$ be a ring and $\sigma: A \to A$ a ring endomorphism. A generalized skew (or $\sigma$-)derivation of $A$ is an additive  map $d: A \to A$ for which there exists a map $\delta:A \to A$ such that $d(xy)=\delta(x)y+\sigma(x)d(y)$ for all $x,y \in A$. If $A$ is a prime $C^*$-algebra and $\sigma$ is surjective, we determine the structure of  generalized $\sigma$-derivations of $A$ that belong to the cb-norm closure of elementary operators $\El(A)$ on $A$; all such maps are of the form $d(x)=bx+axc$ for suitable elements $a,b,c$ of the multiplier algebra $M(A)$. As a consequence, if an epimorphism $\sigma: A \to A$ lies in the cb-norm closure of $\El(A)$, then $\sigma$ must be an inner automorphism. We also show that these results cannot be extended even to relatively well-behaved non-prime $C^*$-algebras like $C(X,\mathbb{M}_2 )$.
\end{abstract}

\maketitle

\section{Introduction}

A well-known consequence of Skolem-Noether theorem (see e.g. \cite[Theorem~4.46]{Bre}) is that a finite-dimensional central simple algebra $A$ over a field $\F$ admits only inner derivations and inner automorphisms. This fact can be also proved by observing that all $\F$-linear maps $\phi : A \to A$ are elementary operators, i.e. they can be written as finite sums of two-sided  multiplications $x \mapsto axb$, with $a,b \in A$ (see \cite[Lemma~1.25, Theorem~1.30]{Bre}). Hence, one can ask the following general question:
\begin{problem}\label{prob1}
Under which conditions on a semiprime ring (or an algebra) $R$ are all derivations and/or automorphisms of $R$ that are also elementary operators necessarily inner?
\end{problem}

In order to investigate Problem \ref{prob1} it is sometimes convenient to consider maps $d: R \to R$ that  comprise both (generalized) derivations and automorphisms. One particularly interesting class of such maps $d$ is the following (see e.g. \cite{Lee,LL}):
\begin{definition}\label{gsd}
Let $R$ be a ring and let $\s : R \to R$ be a ring endomorphism. An additive map $d: R \to R$ is called a \textit{generalized $\s$-derivation} (or a \textit{generalized skew derivation}) if there exists a map $\d:R \to R$  such that
\begin{equation}\label{gsder}
d(xy)=\d(x)y+\s(x)d(y)
\end{equation}
for all $x,y \in R$.
\end{definition}
In case when $R$ is semiprime and $\sigma$ is surjective, in \cite{EGI} we considered the problem of determining the structure of generalized $\sigma$-derivations $d:R\to R$ that are also elementary operators. In order to give a description of such maps, we used standard techniques of the theory of rings of quotients. As a consequence of the proof of \cite[Theorem 1.2]{EGI}, which is the main result of that paper, we showed that if $R$ is semiprime and centrally closed, then a derivation $\delta:R \to R$ (resp. a ring epimorphism $\sigma: R \to R$) is an elementary operator if and only if $\delta$ (resp. $\sigma$) is an inner derivation (resp. inner automorphism), thus giving an affirmative answer to Problem \ref{prob1} for this class of rings.  

It is also interesting to consider Problem \ref{prob1} in the setting of $C^*$-algebras. Moreover, when working with $C^*$-algebras, it is well-known that their derivations, automorphisms and elementary operators are completely bounded. This motivates us to consider the following analytic variation of Problem \ref{prob1}:

\begin{problem}\label{prob2}
Under which conditions on a $C^*$-algebra $A$ are all derivations and/or automorphisms of $A$ that admit a cb-norm approximation by elementary operators necessarily inner?
\end{problem}
In our previous work we considered Problem \ref{prob2} only for derivations. More precisely, we showed that all such derivations of $A$ are inner in a case when $A$ is prime  \cite[Theorem~4.3]{Gog1} or central \cite[Theorem~5.6]{Gog1}. This result was further extended in \cite[Theorem~1.5]{Gog2} for unital $C^*$-algebras whose every Glimm ideal is prime. The latter result in particular applies to derivations of local multiplier algebras (see e.g. \cite{AM}), since their Glimm ideals are prime  \cite[Corollary~3.5.10]{AM}. Hopefully, this result might be useful in order to give an answer to Pedersen's problem from 1978, which asks whether all derivations of local multiplier algebras are inner \cite{Ped}.  

\smallskip 

Motivated by these results, in this paper we consider the problem 
of determining the structure of generalized $\sigma$-derivations of $C^*$-algebras $A$, with $\sigma$ surjective, that can be approximated by elementary operators in the cb-norm. 

The main result of this paper is Theorem \ref{main}, where we fully describe the structure of such maps when $A$ is a prime $C^*$-algebra. In particular, if $A$ is prime, we show that if an epimorphism $\s:A \to A$ admits a cb-norm approximation by elementary operators, then $\s$ must be an inner automorphism of $A$ (Corollary \ref{primeinnaut}). The proof of Theorem \ref{main} is given through several steps in Section \ref{results}. 

In Section \ref{remarks} we consider the possible generalization of Theorem \ref{main} and its consequences for $C^*$-algebras that are not necessarily prime. However, this generalization will not be possible even for some well-behaved $C^*$-algebras, like homogeneous $C^*$-algebras, even thought they have only inner derivations (see \cite[Theorem~1]{Spr} for unital case and  \cite[Proposition~3.2]{Gog3} for general case). In fact, we show that for each $n \geq 2$ there is a compact manifold $X_n$ such that the $C^*$-algebra $C(X_n,\MM_n)$ (where $\MM_n$ is the algebra of $n\times n$ complex matrices) admits  outer automorphisms that are simultaneously elementary operators (Proposition \ref{homogcount}). We also give an example of a unital separable $C^*$-algebra $A$ that admits both outer derivations and outer automorphisms that are also elementary operators (Proposition \ref{Ainf}).

\section{Preliminaries}\label{prem}

Let $R$ be a ring. As usual, by $Z(R)$ we denote its centre. By an ideal of $R$ we always mean a two-sided ideal. An ideal $I$ of $R$ is said to be \textit{essential} if $I$ has a non-zero intersection with every other non-zero ideal of $R$. If $R$ is unital, by $R^\times$ we denote the set of all invertible elements of $R$. 

Recall that a ring $R$ is said to be  \textit{semiprime} if for $x \in R$, $xRx=\{0\}$ implies $x=0$. If in addition for $x,y \in R$, $xRy=\{0\}$ implies $x=0$ or $y=0$, $R$ is said to be \textit{prime}.

If $R$ is a semiprime ring by $M(R)$ we denote its multiplier ring (see \cite[Section~1.1]{AM}). Note that $M(R)$ is also a semiprime ring and that $R$ is prime if and only if $M(R)$ is prime \cite[Lemma~1.1.7]{AM}. For each $a \in M(R)^\times$ we denote by $\mathrm{Ad}(a)$ the  automorphism $x \mapsto axa^{-1}$. We call such automorphisms inner (when $R$ is unital this coincides with the standard notion of the inner automorphism). 

\begin{remark}\label{dpr} If $R$ is a semiprime ring and $d : R \to R$ a generalized $\s$-derivation (Definition \ref{gsd}), then a map $\delta$
is obviously uniquely determined by $d$. Moreover, $\delta$ is a \textit{$\s$-derivation}, that is $\d$ is an additive map that satisfies 
\begin{equation}\label{skewder}
\d(xy)=\d(x)y+\s(x)\d(y)
\end{equation}
for all $x,y \in R$. Indeed, using (\ref{gsder}) and additivity of $d$ and $\s$, for all $x,y,z \in R$ we have  
$$d((x+y)z)=\d(x+y)z+\s(x+y)d(z)=\d(x+y)z+\s(x)d(z)+\s(y)d(z)$$
and 
$$d(xz+yz)=d(xz)+d(yz)=\d(x)z+\s(x)d(z)+\d(y)z+\s(y)d(z).$$
Subtracting these two equations we get the additivity of $\d$. Similarly, subtracting the next two equations
$$d(xyz)=\d(xy)z+
\s(xy)d(z)=\d(xy)z+\s(x)\s(y)d(z),$$
$$d(xyz)=\d(x)yz+\s(x)d(yz)=\d(x)yz + \s(x)\d(y)z+ \s(x)\s(y)d(z)$$
shows (\ref{skewder}). Further, it is now easy to verify that the map $\r:=d-\d$ is a \textit{left $R$-module $\s$-homomorphism}, that is $\rho: R \to R$ is an additive map that satisfies
\begin{equation}\label{skewmodhom}
\rho(xy)=\s(x)\rho(y)
\end{equation}
for all $x,y \in R$. Therefore, every generalized $\s$-derivation can be uniquely decomposed as
$$d=\delta + \rho,$$
where $\delta$ is a $\s$-derivation and $\rho$ is a left $R$-module $\s$-homomorphism. In particular, generalized $\s$-derivations simultaneously generalize $\s$-derivations of $R$ (we get them for $d=\d$) and left $R$-module $\s$-homomorphisms of $R$ (we get them for $\d=0$).
\end{remark}

\begin{example}
The simplest examples of $\sigma$-derivations $\delta : R \to R$ are \textit{inner $\s$-derivations}, i.e. those of the form 
$$\delta(x)=ax-\s(x)a,$$
where $a$ is some element of $M(R)$. Further, any map of the form $$\rho(x)=\sigma(x)a,$$ where $a \in M(R)$, is a left $R$-module $\s$-homomorphism. If $R$ is unital, then all left $R$-module $\s$-homomorphism of $R$ are of this form ($M(R)=R$ in this case).
\end{example}

\smallskip 

Throughout this paper $A$ will be a $C^*$-algebra. Then $M(A)$ has a structure of a $C^*$-algebra and is called the \textit{multiplier algebra} of $A$. It is well-known (and easily checked) that $A$, as a ring, is semiprime. As usual, by $\CB(A)$ we denote the set of all completely bounded maps $\phi : A \to A$ (see e.g. \cite{Paul}). For  $S \subseteq \CB(A)$ we denote by $\overline{\overline{S}}_{cb}$ the cb-norm closure of $S$.  

The most prominent class of completely bounded maps on $A$ are \textit{elementary operators}, i.e. those that can be expressed as finite sums of \textit{two-sided multiplications} $M_{a,b} : x \mapsto
axb$, where $a$ and $b$ are elements of $M(A)$. We denote the set of all elementary operators on $A$ by $\El(A)$. It is well-known that elementary operators on $C^*$-algebras are completely bounded. In fact, we have the following estimate for their cb-norm:
\begin{equation}\label{eq:Haagest}
\left\|\sum_{i} M_{a_i,b_i}\right\|_{cb}\leq \left\|\sum_i a_i \otimes b_i\right\|_h,
\end{equation}
where $\|\cdot\|_h$ is the Haagerup tensor norm on the algebraic
tensor product $M(A) \otimes M(A)$, i.e.
\[
\|t\|_h = \inf \left\{\left\|\sum_{i} a_ia_i^*\right\|^{\frac{1}{2}}\left\|\sum_{i} b_i^*b_i\right\|^{\frac{1}{2}} \ : \ t=\sum_{i} a_i \otimes b_i\right\}.
\]

By inequality (\ref{eq:Haagest}) the mapping
\[
(M(A) \otimes M(A), \|\cdot\|_h) \to (\El(A), \|\cdot\|_{cb}) \quad \mbox{given by} \quad \sum_{i} a_i \otimes b_i \mapsto \sum_{i} M_{a_i,b_i}.
\]
is a well-defined contraction. Its continuous extension to the completed Haagerup tensor product $M(A) \otimes_h M(A)$ is known as a \emph{canonical contraction} from $M(A) \otimes_h M(A)$ to $\CB(A)$ and is denoted by $\Theta_A$. We have the following result (see \cite[Proposition 5.4.11]{AM}):

\begin{theorem}[Mathieu]\label{thetaiso}
$\Theta_A$ is isometric if and only if $A$ is a prime $C^*$-algebra.
\end{theorem}
If $A$ is not a prime note that $\Theta_A$ is not even injective. Indeed, in this case there are non-zero elements $a,b \in A$ such that $aAb=\{0\}$. Then $a\otimes b$ defines a non-zero tensor in $M(A)\otimes M(A)$ but $M_{a,b}=0$. For a unital but not necessarily prime $C^*$-algebra one can construct a central Haagerup tensor product $A \otimes_{Z,h} A$ and consider the induced contraction $\Theta_{A}^Z : A \otimes_{Z,h} A \to \CB(A)$ (the questions when $\Theta_{A}^Z$ is isometric or injective were treated in \cite{Som, AST1, AST2}).

\section{Results}\label{results}

We begin this section by stating the main result of this paper:
\begin{theorem}\label{main} Let $A$ be a prime $C^*$-algebra and let $d: A \to A$ be a generalized $\s$-derivation, with $\s$ surjective. The following conditions are equivalent:
\begin{itemize}
\item[(i)] $d\in \overline{\overline{\El(A)}}_{cb}$.
\item[(ii)] Either $d$ is a left multiplication implemented by some element of $M(A)$ or $\s$ is an inner automorphism of $A$. 
In the latter case, if  $d=\delta+\rho$ is a decomposition  as in Remark \ref{dpr}, then $\delta$ is an inner $\s$-derivation and $\rho$ is a right multiplication of $\sigma$ by some element of $M(A)$.
\end{itemize}
\end{theorem}

\begin{remark}\label{leftmul}
\begin{itemize}
\item[(i)]
Note that any left multiplication $d=M_{l,1}$, where $l\in M(A)$, is a generalized $\s$-derivation with respect to any ring endomorphism $\s : A \to A$. Indeed, for any such $\s$, let $\delta(x)=lx-\s(x)l$ and $\rho(x)=\s(x)l$. Then obviously $d=\d + \rho$, so  in this case we cannot say anything about the epimorphism $\s$ and the corresponding maps $\d$ and $\rho$. 
\item[(ii)] If $d$ is not a left multiplication, then following the second case of part (ii) of Theorem \ref{main} we have 
$$\s=\Ad(a), \quad \delta(x)=bx-\s(x)b \quad \mbox {and} \quad \rho(x)=\s(x)b',$$ for some $a \in M(A)^\times$ and  $b,b' \in M(A)$, $b \neq b'$. In particular, $d$ is of the form
$$d(x)=bx+axc,$$
where $c:=a^{-1}(b'-b)$. 
\end{itemize}
\end{remark}

\begin{remark}
In the sequel of this section we assume that $A$ is an infinite-dimensional prime $C^*$-algebra, since otherwise \cite[Theorem~6.3.8]{Mur} and the primeness of $A$ would imply that $A$ is isomorphic to the matrix algebra $\MM_n$ for some non-negative integer $n$. Then every linear map $\phi :A \to A$ is an elementary operator (see e.g. \cite[Lemma~1.25]{Bre}), so Theorem \ref{main} is just a simple consequence of \cite[Theorem 1.2]{EGI} (the maximal right ring of quotients $Q_{mr}(A)$ in this case coincides with $A$).
\end{remark}

For the proof of Theorem \ref{main} we will need some auxiliary results. 
We start with the following:

\begin{lemma}\label{lemMsigma}
Let $A$ be a prime $C^*$-algebra. Suppose that $\s : A \to A$ is a ring epimorphism for which there exists a non-zero element $a \in M(A)$ such that
\begin{equation}\label{Ms}
ax=\sigma(x)a \quad \forall x \in A.
\end{equation}
Then $a$ is invertible in $M(A)$, so that $\s=\Ad(a)$ is an inner automorphism of $A$.
\end{lemma}

Before proving Lemma \ref{lemMsigma} recall from \cite{AM} (see also \cite{BMM}) that an \textit{essentially defined double centralizer} on a semiprime ring $R$ is a triple 
$(\cL,\cR,I)$, where $I$ is an essential ideal of $R$, $\cL : I \to R$ is a left $R$-module homomorphism, $\cR: I \to R$ is a right $R$-module homomorphism such that $\cL(x)y=x\cR(y)$ for all $x \in I$.
One can form the \emph{symmetric ring of quotients} $Q_{s}(R)$ which is characterized (up to isomorphism) by the following properties:
\begin{itemize}
\item[(i)]
$R$ is a subring of $Q_{s}(R)$;
\item[(ii)] for any $q\in Q_{s}(R)$ there is
an essential ideal $I$ of $R$ such that $qI+Iq\subseteq R$;
\item[(iii)] if $0\neq q\in Q_{s}(R)$ and $I$ is an essential ideal of $R$, then $qI\neq0$ and $Iq\neq 0$;
\item[(iv)] for any essentially defined double centralizer $(\cL,\cR,I)$ on $R$ there exists $q \in Q_s(R)$ such that $\cL(x)=qx$ and $\cR(x)=xq$ for all $x \in I$.
\end{itemize}
In a case when $R=A$ is a $C^*$-algebra, $Q_s(A)$ has a natural structure as a unital complex $*$-algebra, whose involution is positive definite. An element $q \in Q_s(A)$ is called \emph{bounded}
if there is $\lambda \in \R_+$ such that $q^*q \leq \lambda 1$, in a sense that there is a finite number of elements $q_1,\ldots, q_n\in Q_s(A)$ such that 
$$q^*q + \sum_{i=1}^n q_i^*q_i=\lambda 1.$$ 
The set $Q_b(A)$ of all bounded elements of $Q_s(A)$ has a pre-$C^*$-algebra structure with respect to the norm 
$$\|q\|^2=\inf\{\lambda \in \R_+ : \ q^*q \leq \lambda 1\},$$
which clearly extends the norm of $A$. One can easily check that an element  $q \in Q_s(A)$ is bounded if and only if it can be represented by a bounded (continuous) essentially defined double centralizer (see \cite[p.~57]{AM}). We call $Q_b(A)$ the \emph{bounded symmetric algebra of quotients} of $A$ and its completion $\M(A)$ the \emph{local multiplier algebra} of $A$. Note that $\M(A)$ has a structure of a $C^*$-algebra as a completion of a pre-$C^*$-algebra.

\begin{proof}[Proof of Lemma \ref{lemMsigma}]
First note that a non-zero element $a \in M(A)$ that satisfies (\ref{Ms}) cannot be a zero-divisor. Indeed, if there exists $x \in M(A)$ such that $ax=0$ then
for each $y \in A$ we have $ayx=\sigma(y)ax=0,$ so that $aAx=\{0\}$. Since $A$ is prime and since $A$ is an essential ideal of $M(A)$, $a \neq 0$ implies  $x=0$. Similarly, if $xa=0$ then for all $y\in A$ we have $x \sigma(y)a=xay=0$. Since $\sigma$ is surjective, this is equivalent to $xAa=\{0\}$, so the primeness of $A$ again implies $x=0$.

We now show that $a$ is invertible in $M(A)$. Since $M(A)$ is a unital $C^*$-subalgebra of $\M(A)$, we have $M(A)^\times = M(A) \cap \M(A)^\times$, so it suffices to show that $a$ is invertible in $\M(A)$.
In order to do this, first note that $aA$ is a non-zero ideal of $A$, hence essential, since $A$ is prime. Indeed, since $\s$ is surjective, we have $A=\s(A)$, hence
$$
AaA=\left\{\sigma(x)ay : \ x,y \in A\right\} = \left\{axy  : \ x,y \in A \right\}\subseteq aA.
$$
In particular, for $\alpha \in \mathbb{C}$ and $x \in A$ we have $(\sigma(\alpha x)-\alpha \sigma(x))aA=\{0\}$, which implies $\sigma(\alpha x)=\alpha \sigma(x)$. Therefore $\sigma$ is a linear map, hence an algebra epimorphism.

We define maps $\cL,\cR : aA \to A$ by
$$\cL(ax)=\sigma(x) \quad \mbox{ and } \quad  \cR(ax)=x.$$
That the maps $\cL$ and $\cR$ are well-defined follows from the fact that $a$ is not a (left) zero-divisor. Clearly, $\cR$ is a right $A$-module homomorphism. Next, for $x,y \in A$ we have
$$\cL(\s(x)ay)=\cL(axy)=\s(xy)=\s(x)\s(y) = \s(x)\cL(ay)$$
and 
$$\cL(ax)ay=\s(x)ay=axy=ax\cR(ay).$$
This shows that $(\cL,\cR,aI)$ is an essentially defined double centralizer on $A$. Since $A$ is prime, by \cite[Corollary 2.2.15]{AM} all essentially defined double centralizers are automatically continuous.
In particular, there exists an element $b \in Q_b(A)\subseteq \M(A)$ such that
$$\s(x)=\cL(ax)=axb=\s(x)ab \quad \mbox{ and } \quad x=\cR(ax)=bax$$ 
for all $x \in A$. Since $\s(A)=A$, this is equivalent to $A(1-ab)=\{0\}$ and $(1-ba)A=\{0\}$. Hence,
$a$ is invertible in $\M(A)$ and $a^{-1}=b$. This completes the proof.
\end{proof}

Recall from \cite[Definition 3.2]{Smi} that a sequence $(a_n)$ in an infinite-dimensional $C^*$-algebra $B$ such that the series $\sum_{n=1}^\infty a_n^*a_n$
is norm convergent is said to be \textit{strongly independent} if
for every sequence  $(\alpha_n) \in \ell^2$, equality $\sum_{n=1}^\infty \alpha_n a_n=0$ implies $\alpha_n=0$ for all $n \in \N$.

The next fact can be deduced from \cite[Proposition 1.5.6]{BM}, \cite[Lemma 4.1]{Smi} and
\cite[Lemma 2.3]{ASS}.
\begin{remark}\label{rep} Let $B$ be an infinite-dimensional $C^*$-algebra.
\begin{itemize}
\item[(i)] Every tensor $t \in B \otimes_h B$ has a representation as a
convergent series $t=\sum_{n=1}^\infty a_n \otimes b_n,$ where $(a_n)$ and
$(b_n)$ are sequences in $B$ such that the series $\sum_{n=1}^\infty a_n a_n^*$
and $\sum_{n=1}^\infty b_n^*b_n$ are norm convergent. Moreover, the sequence $(b_n)$ can be chosen to be strongly independent.
\item[(ii)] If $t=\sum_{n=1}^\infty a_n \otimes b_n$ is a representation of $t$
as above, with $(b_n)$ strongly independent, then $t=0$ if and only if $a_n=0$
for all $n \in \N$. 
\end{itemize}
\end{remark}

\begin{corollary}\label{corprimeinj}
Let $A$ be a prime $C^*$-algebra and suppose that $(a_n)$, $(b_n)$ are sequences in $M(A)$ such that the series $\sum_{n=1}^\infty a_n a_n^*$
and $\sum_{n=1}^\infty b_n^*b_n$ are norm convergent, with $(b_n)$ strongly independent. If 
\begin{equation}\label{inj}
\sum_{n=1}^\infty a_n x b_n=0
\end{equation}
for all $x \in A$, then $a_n=0$ for all $n \in \N$.
\end{corollary}
\begin{proof}
If $t:=\sum_{n=1}^\infty a_n \otimes b_n \in M(A) \otimes_h M(A)$, then (\ref{inj}) is equivalent to $\Theta_A(t)=0$. Since $A$ is prime, by Theorem \ref{thetaiso} $\Theta_A$ is isometric (hence injective), so $t=0$. The claim now follows from part (ii) of Remark \ref{rep}.
\end{proof}

\begin{proposition}\label{propepi}
Let $A$ be a prime $C^*$-algebra and let $\s : A \to A$ be a ring  epimorphism. If  $\r : A \to A$ is a non-zero left $A$-module  $\s$-homomorphism, then the following conditions are equivalent:
\begin{itemize}
\item[(i)] $\r \in \overline{\overline{\El(A)}}_{cb}$.
\item[(ii)] There are elements $a,p \in M(A)$, with $a$ invertible and $p \neq 0$, such that $\s=\Ad(a)$ and
$$\r(x)=\s(x)p=axa^{-1}p$$ 
for all $x \in A$.
\end{itemize}
\end{proposition}
\begin{proof} 
Since $A$ is prime, by Theorem \ref{thetaiso} the canonical contraction $\Theta_A : M(A) \otimes_h M(A) \to \CB(A)$ is isometric. In particular, 
the image of $\Theta_A$ is closed in the cb-norm so $\overline{\overline{\El(A)}}_{cb}$ coincides with the image of $\Theta_A$.
Hence, there is a tensor $t \in M(A) \otimes_h M(A)$ such that 
$\rho=\Theta_A(t)$. By Remark \ref{rep}, we can write
$t=\sum_{n=1}^\infty a_n \otimes b_n$, where $(a_n)$ and
$(b_n)$ are sequences in $M(A)$ such that the series $\sum_{n=1}^\infty a_n a_n^*$ and $\sum_{n=1}^\infty b_n^*b_n$ are norm convergent, with $(b_n)$ strongly independent. Then (\ref{skewmodhom}) implies
$$\sum_{n=1}^\infty (a_n x-\s(x)a_n)yb_n=0$$
for all $x,y\in A$. By Corollary \ref{corprimeinj} we have
\begin{equation}\label{ans}
a_nx = \s(x)a_n
\end{equation}
for all $n \in \N$. Since $\rho$ is non-zero, there is $n_0 \in \N$ such that $a_{n_0} \neq 0$. By Lemma \ref{lemMsigma} $a:=a_{n_0}$ is invertible
in $M(A)$. Hence $\s=\Ad(a)$ is an inner automorphism of $A$. Finally, if 
$p:=\sum_{n=1}^\infty a_n b_n \in M(A)$, using (\ref{ans}) we get
$$
\rho(x)=\sum_{n=1}^\infty a_n x b_n = \s(x) \left(\sum_{n=1}^\infty a_nb_n \right) = \s(x)p = axa^{-1}p. 
$$
\end{proof}

As a direct consequence of Proposition \ref{propepi} we get:
\begin{corollary}\label{primeinnaut}
If $A$ is a prime $C^*$-algebra then every ring epimorphism $\s : A \to A$ that lies in $\overline{\overline{\El(A)}}_{cb}$ must be an inner automorphism of $A$. 
\end{corollary}

The next fact can be deduced from the proof of \cite[Theorem 4.3]{Gog1}. For completeness, we include a proof.
\begin{lemma}\label{lemteh}
Let $B$ be a unital infinite-dimensional $C^*$-algebra and let $f,g,h : B \to B$ be any functions with $f \neq 0$. Suppose that for all $x \in B$ we have the following 
equality of tensors in $B \otimes_h B$
\begin{equation}\label{funten}
f(x) \otimes 1 = \sum_{n=1}^\infty (a_n g(x)+h(x)a_n) \otimes b_n,
\end{equation}
where $(a_n)$ and $(b_n)$ are sequences in $B$ such that the series $\sum_{n=1}^\infty a_n a_n^*$
and $\sum_{n=1}^\infty b_n^*b_n$ are norm convergent, with $(b_n)$ strongly independent. Then there is a non-zero element $b \in B$ such that 
\begin{equation}\label{funten2}
f(x)= bg(x)+h(x)b
\end{equation}
for all $x \in B$. 
   
\end{lemma}

\begin{proof}
Choose $x_0 \in B$ such that $f(x_0)\neq 0$ and let  $\varphi \in B^*$ be an arbitrary bounded linear functional such that
$\varphi(f(x_0))\neq 0$. If for $x=x_0$ we act on the equality (\ref{funten}) with
the right slice map $R_\varphi :  B \otimes_h B \to B$, $R_\varphi : a \otimes b \mapsto \varphi(a)b$ (see e.g. \cite[Section 4]{Smi}), we obtain
\begin{equation}\label{pom}
\varphi(f(x_0)) \cdot 1=\sum_{n=1}^\infty \varphi(a_n g(x_0) + h(x_0)a_n)
b_n.
\end{equation}
For $n \in \N$ let
$$\alpha_n:=\frac{\varphi(a_ng(x_0)+h(x_0)a_n)}{\varphi(f(x_0))}.$$
Note that $(\alpha_n) \in \ell^2$, since all bounded linear functionals on $C^*$-algebras are completely bounded (see e.g. \cite[Proposition 3.8]{Paul}) and the series
$\sum_{n=1}^\infty (a_ng(x_0)+h(x_0)a_n)(a_ng(x_0)+h(x_0)a_n)^*$ is norm convergent. Then (\ref{pom}) can be rewritten as $\sum_{n=1}^\infty \alpha_n b_n=1$, so  by (\ref{funten}) we have 
$$\sum_{n=1}^\infty (\alpha_n f(x)-a_n g(x) - h(x)a_n)\otimes b_n=0$$
for all $x \in B$. Consequently, since $(b_n)$ is strongly independent, Remark \ref{rep} (ii) implies that
$$
\alpha_nf(x)= a_n g(x) + h(x)a_n
$$
for all $n \in \N$ and $x \in B$. Since $\sum_{n=1}^\infty \alpha_n b_n=1$, there is some $n_0 \in \N$ such that
$\alpha_{n_0} \neq 0$. If $b:=(1/\alpha_{n_0})a_{n_0}$, then the above equation is obviously equivalent to (\ref{funten2}). Also, $b \neq 0$ since $f \neq 0$.
\end{proof}

\begin{proof}[Proof of Theorem \ref{main}] 

(ii) $\Longrightarrow$ (i). This is trivial (see also Remark \ref{leftmul}).

(i) $\Longrightarrow$ (ii). Assume that $d\in \overline{\overline{\El(A)}}_{cb}$ and that $d\neq M_{l,1}$ for all $l \in M(A)$. In particular $d \neq 0$. Using the same arguments from the beginning of the proof of Proposition \ref{propepi}, we see that there is a tensor $t \in M(A) \otimes_h M(A)$ such that 
$d=\Theta_A(t)$. By Remark \ref{rep}, we can write
$t=\sum_{n=1}^\infty a_n \otimes b_n$, where $(a_n)$ and
$(b_n)$ are sequences of $M(A)$ such that the series $\sum_{n=1}^\infty a_n a_n^*$ and $\sum_{n=1}^\infty b_n^*b_n$ are norm convergent, with $(b_n)$ strongly independent. Using (\ref{gsder}) for all $x,y \in A$ we get
$$\d(x)y  = \sum_{n=1}^\infty (a_n x - \s(x)a_n)y b_n,$$
or equivalently
\begin{equation}\label{Thetaeq}
\Theta_A (\d(x) \otimes 1)) = \Theta_A \left(\sum_{n=1}^\infty (a_n x - \s(x)a_n) \otimes b_n\right).
\end{equation}
Again, since $\Theta_A$ is isometric (hence injective), (\ref{Thetaeq}) is equivalent to the equality 
$$\d(x) \otimes 1 = \sum_{n=1}^\infty (a_n x - \s(x)a_n) \otimes b_n$$
of tensors in $M(A) \otimes_h M(A)$ for all $x \in A$. If $\delta=0$, then $d$ must be a non-zero left $A$-module $\s$-homomorphism of $A$ (see Remark \ref{dpr}) so the claim follows directly from Proposition \ref{propepi}. If $\delta\neq 0$,
Lemma \ref{lemteh} implies that there is a non-zero element $b \in M(A)$ such that
$$\delta(x)=bx-\s(x)b$$
for all $x \in A$. If we decompose $d=\delta+\rho$ as in Remark \ref{dpr}, the map $\rho' : A \to A$ defined by  $$\rho'(x):=\rho(x)-\s(x)b=d(x)-bx$$ is obviously a left $A$-module $\s$-homomorphism of $A$ that lies in $\overline{\overline{\El(A)}}_{cb}$ (since $d$ does). Since, by assumption, $d$ is not a left multiplication, $\rho'$ is non-zero. Hence, by Proposition \ref{propepi} there are elements $a,p \in M(A)$ with $a$ invertible and $p \neq 0$ such that $\s=\Ad(a)$ and
$\rho'(x)=\s(x)p$ for all $x \in A$. In particular, if we put $b':=b+p$, we get $\rho(x)=\s(x)b'$, which completes the proof.
\end{proof}

\section{Counterexamples and further remarks}\label{remarks}

In \cite{Gog1, Gog2} we considered derivations of unital $C^*$-algebras $A$ that lie in $\overline{\overline{\El(A)}}_{cb}$. We showed that all such derivations are inner in a case when $A$ is prime \cite[Theorem~4.3]{Gog1} or central \cite[Theorem~5.6]{Gog1}, or more generally, when $A$ is a unital $C^*$-algebra whose every Glimm ideal is prime \cite[Theorem~1.5]{Gog2}. 

In light of this, it is natural to ask if one can extend Corollary \ref{primeinnaut} in its original form (and consequently Theorem \ref{main}) for  similar classes of $C^*$-algebras. However, this will not be possible, even for relatively well-behaved $C^*$-algebras like homogeneous $C^*$-algebras. In fact, we will now show that for all $n \geq 2$, a $C^*$-algebra $A_n=C(PU(n),\MM_n)$, where $PU(n)=U(n)/\mathbb{S}^1$ is the projective unitary group, admits outer automorphisms which are simultaneously elementary operators on $A_n$ (Proposition \ref{homogcount}).

\smallskip

In order to show this, first suppose that  $A$ is a general separable $n$-homogeneous $C^*$-algebra (i.e. all irreducible representations of $A$ have the same finite dimension $n$). Then by \cite[Theorem~4.2]{Kap} the primitive spectrum $X:=\mathrm{Prim}(A)$ is a (locally compact) Hausdorff space and by a well-known theorem of Fell \cite[Theorem~3.2]{Fell} and Tomiyama-Takesaki \cite[Theorem~5]{TT} there is a locally trivial bundle $\CE $ over $X$ with fibre $\MM_n$ and structure group $\mathrm{Aut}^*(\mathbb{M}_n)\cong PU(n)$ such that $A$ is isomorphic to the
$C^*$-algebra $\Gamma_0(\CE )$ of continuous sections of $\CE $ that vanish at infinity. Moreover, any two such algebras $A_i=\Gamma_0(\CE_i )$ with primitive
spectra $X_i$ ($i =1,2$)
are isomorphic if and only if there is a homeomorphism
$f : X_1 \to X_2$ such that $\CE _1 \cong f^*(\CE _2)$ (the pullback bundle)
as bundles over $X_1$ (see \cite[Theorem~6]{TT}). Thus, we may identify $A$ with $\Gamma_0(\CE )$. Further, by Dauns-Hofmann theorem \cite[Theorem A.34]{RW} we can identify $Z(A)$ with $C_0(X)$ and $Z(M(A))$ with $C_b(X)$.

Let us denote by  $\mathrm{Aut}_{Z}^*(A)$ the set of all $Z(M(A))$-linear $*$-automorphisms of $A$ (i.e. $\s(x^*)=\s(x)^*$ and $\s(zx)=z\s(x)$ for all $z \in Z(M(A))$ and $x \in A$) and by $\mathrm{InnAut}^*(A)$ the set of automorphisms $\s$ of $A$ that are of the form $\s=\mathrm{Ad}(u)$, for some unitary element $u \in M(A)$. Obviously $\mathrm{InnAut}^*(A) \subseteq \mathrm{Aut}_{Z}^*(A)$. It is a very interesting (and non-trivial) problem to describe when do we have $\mathrm{Aut}_{Z}^*(A)=\mathrm{InnAut}^*(A)$ (see e.g. \cite{MsSmi,Lan,PR,PRT} for some results regarding this question). In particular, we have the following consequence of \cite[Theorem~2.1]{PR} (see also \cite[2.19]{PR}):

\begin{theorem}[Phillips-Raeburn]\label{ExSec}
If $A=\Gamma_0(\CE )$ is a separable $n$-homogeneous $C^*$-algebra with primitive spectrum $X$, we have the exact sequence 
$$
0  \longrightarrow  \mathrm{InnAut}^*(A) \longrightarrow \mathrm{Aut}^*_{Z} (A) \stackrel{\eta}\longrightarrow \check{H}^2(X;\Z)
$$
of abelian groups, where $\check{H}^2(X;\Z)$ is the second integral \v Cech cohomology group of $X$. Further, the image of $\eta$ is contained in the torsion subgroup of $\check{H}^2(X;\Z)$.
\end{theorem} 

\begin{proposition}\label{homogcount}
For $n  \geq 2$ let  $A_n=C(PU(n),\MM_n)$. Then every derivation of $A_n$ is inner, but $A_n$ admits an outer automorphism that is also an elementary operator.  
\end{proposition}

\begin{remark}\label{homogidpres}
A $C^*$-algebra $A$ is said to be \textit{quasicentral} if any element $a \in A$ can be decomposed as $a=zb$ for some $b \in A$ and $z \in Z(A)$ (see \cite{Del,Arc,Gog1} for other characterizations of such algebras). If every closed ideal of $A$ is quasicentral, note that any $Z(A)$-linear map $\phi : A \to A$ preserves all closed ideals of $A$ (i.e. $\phi(I)\subseteq I$ for any such ideal $I$). 
Indeed, since any $a \in I$ can be decomposed as  $a=zb$, with $z \in Z(I)$ and $b \in I$ and since $Z(I)\subseteq Z(A)$, we have $\phi(a)=\phi(zb)=z\phi(b) \in I$. 

This observation in particular applies to $n$-homogeneous $C^*$-algebras $A\cong \Gamma_0(\CE )$. Indeed, using the fact that the bundle $\CE$ is locally trivial one can easily check that $n$-homogeneous $C^*$-algebras are quasicentral. 
Also, every closed  ideal of an $n$-homogeneous $C^*$-algebra is also an $n$-homogeneous $C^*$-algebra, hence quasicentral.  Further, if $A$ is unital (and $n$-homogeneous), every bounded $Z(A)$-linear map $\phi : A \to A$ is an elementary operator on $A$. This follows directly from the above observation and Magajna's theorem \cite[Theorem~1.1]{Mag}. Therefore, for every unital $n$-homogeneous $C^*$-algebra $A$ we have $\mathrm{Aut}^*_{Z} (A)\subseteq \El(A)$.
\end{remark}

\begin{proof}[Proof of Proposition \ref{homogcount}]
That all derivations of $A_n$ are inner follows from \cite[Theorem~1]{Spr}. On the other hand, by \cite[Section~IV]{KR} $A_n$ admits an automorphism $\s \in \mathrm{Aut}^*_{Z} (A_n) \setminus  \mathrm{InnAut}^*(A_n)$ (note that $\check{H}^2(PU(n);\Z)=\Z_n$, so all elements of $\check{H}^2(PU(n);\Z)$ are torsion elements). By Remark \ref{homogidpres}, $\s\in \El(A_n)$. Suppose that $\s=\mathrm{Ad}(a)$ for some $a \in A_n^\times$.  Then, since $\s$ is $*$-preserving, for all $x \in A_n$ we have
$$ax^*a^{-1}=\s(x^*)=\s(x)^* =(a^{*})^{-1}x^*a^*.$$
Hence $a^{*}a \in Z(A_n)$, so $|a|=\sqrt{a^*a} \in Z(A_n)$. Therefore, if $u:=|a|^{-1}a$, then $u$ is a unitary element of $A_n$ and $\sigma=\mathrm{Ad}(u)\in \mathrm{InnAut}^*(A_n)$; a contradiction. 
\end{proof}

On the other hand, if $A_n=C(PU(n), \MM_n )$ as before, every $\sigma \in  \mathrm{Aut}^*_{Z} (A_n)$ is implemented by some unitary element of $Q_b(A)$. This follows from the following fact:

\begin{proposition}\label{loccontr}
Let $A=\Gamma_0(\CE )$ be a separable $n$-homogeneous $C^*$-algebra whose primitive spectrum $X$ is locally contractable. Then for every $\sigma \in\mathrm{Aut}^*_{Z} (A)$ there is a unitary element $u \in Q_b(A)$ such that $\s=\mathrm{Ad}(u)$.
\end{proposition}

\begin{proof}
We first show that there is a dense open subset $U$ of $X$ such that $\check{H}^2(U;\Z)=0$. This can be shown by using the similar arguments as in the proof of \cite[Lemma~3.1]{Gog4}. Indeed, let $\mathfrak{F}$ be a collection of all families $\mathcal{V}$ consisting of mutually disjoint open contractable subsets of $X$. We use the standard set-theoretic inclusion for
partial ordering. If $\mathfrak{C}$ is a chain in $\mathfrak{F}$, then obviously $\bigcup \mathfrak{C}$ is an upper bound of $\mathfrak{C}$ in $\mathfrak{F}$. Therefore, applying
Zorn's lemma, we obtain a maximal family $\mathcal{M}$ in $\mathfrak{F}$. Let $U$ be the union of all members of $\mathcal{M}$. Since $U$ is a disjoint union of contractable spaces, we have $\check{H}^2(U;\Z)=0$.
Since $X$ is locally contractable (and regular), the maximality of $\mathcal{M}$ implies that $U$ is a dense (evidently open) subset of $X$. 

Now let $I:=\Gamma_0(\CE|_U )$, where $\CE|_U $ is a restriction bundle of $\CE$ to $U$. Since $U$ is dense in $X$, $I$ is an essential closed ideal of $A$. If $\s \in \mathrm{Aut}^*_{Z} (A)$, by Remark \ref{homogidpres} we have $\s(I)\subseteq I$, so $\s|_I \in \mathrm{Aut}^*_{Z} (I)$. Since $\PP(I)=U$ and $\check{H}^2(U;\Z)=0$, by Theorem \ref{ExSec} there is a unitary element $u \in M(I)$ such that $\s(x)=uxu^*$ for all $x \in I$. Since $I$ is an essential ideal of $A$, we also have $\s(x)=uxu^*$ for all $x \in A$. If we define $\cL_u, \cR_u : I \to A$ by 
$\cL_u(x)=xu$ and $\cR_u(x)=ux$, then obviously $(\cL_u,\cR_u,I)$ is a bounded 
essentially defined double centralizer of $A$, so $u \in Q_b(A)$ and $\s=\mathrm{Ad}(u)$.
\end{proof}

\begin{problem}
Can we omit the assumption of local contractibility of the space $X$ in  Proposition \ref{loccontr}, that is if $A$ is a general separable $n$-homogeneous $C^*$-algebra, are all automorphisms $\sigma \in \mathrm{Aut}^*_{Z} (A)$ of the form $\s=\mathrm{Ad}(u)$ for some unitary  $u \in Q_b(A)$? 
\end{problem}

\smallskip 

In \cite[Example~6.1]{Gog1} we gave an example of a unital separable $C^*$-algebra $A$ which admits outer derivations that are also elementary operators. We now show that the same $C^*$-algebra admits outer automorphisms that are also elementary operators:

\begin{proposition}\label{Ainf} Let $B:=C([1,\infty],\MM_2)$ be a $C^*$-algebra that consists of all continuous functions from the extended interval $[1,\infty]$ to the $C^*$-algebra $\MM_2$. If $A$ is a $C^*$-subalgebra of $B$ that consists of all $a \in B$ such that
$$a(n)=\left[ \begin{array}{cc}
               \lambda_n(a) & 0 \\
              0 & \lambda_{n+1}(a)\\
               \end{array} \right] \qquad   (n \in \N)$$
for some convergent sequence $(\lambda_n(a))$ of complex
numbers, then $A$ admits an outer automorphism which is also an elementary operator on $A$.
\end{proposition}

\begin{proof} Let $f:[1, \infty] \to \C$ be a continuous function such the series $\sum_{n=1}^\infty f(n)$ does not converge and such that the range of $f$ is a subset of the imaginary axis. We define an element $b\in B$  by
$$b:=\left[ \begin{array}{cc}
               f & 0 \\
              0 & 0 \\
               \end{array} \right].$$
It was observed in \cite[Section 6]{Gog1} that $\d:=\mathrm{ad}(b)$ defines an outer derivation of $A$ which lies in the operator norm closure of the space of all inner derivations of $A$. Since $b^*=-b$, $\d$
is a $*$-derivation (i.e.~$\d(x^*)=\d(x)^*$ for all $x \in A$). Hence $\sigma:=\exp(\d)$ defines a $*$-automorphism of $A$ (see e.g.~\cite[Section 4.3]{AM}). Note that $\sigma=\mathrm{Ad}(u)$, where $$u:=\exp b=\left[ \begin{array}{cc}
               \exp f & 0 \\
              0 & 1 \\
               \end{array} \right]$$ is a unitary element of $B$. Since the exponential map is continuous, and since $\d$ in fact lies in the operator norm closure of the space of all inner $*$-derivations of $A$, we conclude that $\sigma$ lies
               in the operator norm closure of the set of all inner $*$-automorphisms of $A$.

\textit{Claim 1.} $\sigma$ is an outer automorphism of $A$.

On the contrary, suppose that there exists an invertible element $a \in A$ such that $\sigma=\mathrm{Ad}(a)$. Then $u^*ax=xu^*a$ for all $x \in A$. Since $$J:=\{a \in A \ : \ a(n)=0 \ \mathrm{for} \ \mathrm{all} \ n \in \N\}$$ defines an closed essential ideal of both $A$ and $B$, we conclude that $u^*a$ is an invertible central element of $B$.
               Hence, there exists an invertible continuous function $\varphi \in C([1, \infty])$ such that $$a =(\varphi \oplus \varphi)u=\left[ \begin{array}{cc}
                 \exp f \cdot \varphi & 0 \\
              0 & \varphi \\
               \end{array} \right].$$
               Since $a \in A$, we conclude that $\varphi(n+1)=(\exp{f(n)})\varphi(n)$, and consequently
\begin{equation}\label{conv}
\varphi(n+1)=\exp \left (\sum_{k=1}^n f(k)\right ) \varphi(1)
\end{equation}
for all $n \in \N$. Finally, since $\lim_{x \to \infty} \varphi(x)$ exists and since $\varphi(1)\neq 0$, (\ref{conv}) implies that the series $\sum_{n=1}^\infty f(n)$ converges, a contradiction.

\textit{Claim 2.} $\sigma$ is an elementary operator on $A$.

As noted, $\sigma$ lies in the operator norm closure of the set of all inner $*$-auto\-morphisms of $A$. In particular, $\sigma$ lies in the operator norm closure of the set of all elementary operators on $A$. But the latter set is closed in the operator norm by \cite[Lemma 6.6]{Gog1}. Hence, $\sigma$ is an elementary operator.
\end{proof}

We end this paper with the following question:

\begin{problem}
Is Corollary \ref{primeinnaut} true for all von Neumann algebras? In particular, if an automorphism $\sigma$ of a von Neumann algebra $A$ is also an elementary operator, is $\sigma$ necessarily an inner automorphism?
\end{problem}

\end{document}